\documentclass[a4paper,reqno,11pt,oneside]{amsart}

\usepackage{amssymb}
\usepackage{amsmath}
\usepackage{amsthm}
\usepackage{a4wide}
\usepackage{tikz}
\usepackage{stix}
\usepackage{color}
\usepackage{graphpap}
\usepackage{epsfig}
\usepackage{psfrag}
\usepackage{graphicx}
\usepackage{subfigure}
\usepackage{tikz-network}
\usetikzlibrary{arrows,snakes,backgrounds,automata,trees,shapes}
\usepackage{multirow}
\usepackage{threeparttable}
\usepackage{float}
\usepackage{rotating}
\usepackage{cite}
\usepackage{bbm}
\usepackage{comment}
\usepackage{url}

\newtheorem{theorem}{Theorem}[section]
\newtheorem{corollary}{Corollary}[section]
\newtheorem{lemma}{Lemma}[section]

\theoremstyle{definition}
\newtheorem{definition}{Definition}[section]

\theoremstyle{remark}
\newtheorem{remark}{Remark}

\begin{document}

\title{Critical thresholds in stochastic rumors on trees}



\date{}

\author[Jhon F. Puerres]{Jhon F. Puerres}  
\address{Jhon F. Puerres. Universidade Federal de Pernambuco, Av. Prof. Moraes Rego, 1235. Cidade Universit\'aria, CEP 50670-901, Recife, PE, Brazil. E-mail: jhon.puerres@ufpe.br}

\author[Valdivino V. Junior]{Valdivino V. Junior} 
\address{Valdivino V. Junior. Universidade Federal de Goias, Campus Samambaia, CEP 74001-970,
Goiânia, GO, Brazil. E-mail: vvjunior@ufg.br}

\author[Pablo M. Rodriguez]{Pablo M. Rodriguez} 
\address{Pablo M. Rodriguez. Universidade Federal de Pernambuco, Av. Prof. Moraes Rego, 1235. Cidade Universit\'aria, CEP 50670-901, Recife, PE, Brazil. E-mail: pablo@de.ufpe.br}


\begin{abstract}
    The vertices of a tree represent individuals in one of three states: ignorant, spreader, or stifler. A spreader transmits the rumor to any of its nearest ignorant neighbors at rate one. At the same rate, a spreader becomes a stifler after contacting nearest-neighbor spreaders or stiflers. The rumor survives if, at all times, there exists at least one spreader. We consider two extensions and prove phase transition results for rumor survival. First, we consider the infinite Cayley tree of coordination number $d+1$, with $d\geq 2$, and assume that as soon as an ignorant hears the rumor, the individual becomes spreader with probability  $p$, or stifler with probability $1-p$. Using coupling with  branching processes we prove that for any $d$ there is a phase transition in $p$ and localize the critical parameter. By refining this approach, we extend the study to an inhomogeneous tree with hubs of degree $d+1$ and other vertices of degree at most $k=o(d)$. The purpose of this extension is to illustrate the impact of the distance between hubs on the dissemination of rumors in a network. To this end, we assume that each hub is, on average, connected to $\alpha (d+1)$ hubs, with $\alpha\in (0,1]$, via paths of length $h$. We obtain a phase transition result in $\alpha$ in terms of $d,k,$ and $h$, and we show that in the case of $k=\Theta(\log d)$ phase transition occurs iff $h \lesssim \Theta( \log d / (\log \log d))$.
\end{abstract}


\maketitle

\section{Introduction}{ 
Nonlinear dynamics \cite{1aaa,1bbb} and non-equilibrium statistical physics \cite{1ccc,1ddd} provide a natural framework for modeling rumor spreading in complex networks, since these processes are inherently stochastic and sensitive to fluctuations \cite{1eee,1fff}.} Rumor propagation can be viewed as a dynamical system on graphs, where local interactions between agents give rise to emergent macroscopic behaviors such as thresholds, cascades, or extinction. While similar to epidemic processes, rumor dynamics exhibit distinctive nonlinear features, including forgetting and stifling. Tools from the theory of fluctuations and random processes, such as branching processes, Markov chains, and stochastic differential equations, make it possible to quantify variability, critical thresholds, and phase transitions in rumor dynamics. These methods bridge applied mathematics and physics, showing how microscopic randomness and network topology together shape global spreading patterns. For published works in this direction, we refer the reader to \cite{daley_nature} for a classical stochastic rumor model; \cite{moreno-PhysA2007} on complex networks and rumor dynamics; \cite{MNP-PRE2004} on threshold behaviors; \cite{Pastor} for a comprehensive review of epidemic and rumor spreading in networks from a statistical physics perspective; and \cite{zhang,zhou}, along with the references therein, for reviews of studies on rumor modeling and control, with emphasis on nonlinear behavior.

To our knowledge, some of the first mathematical models for rumor transmission were proposed in \cite{DK,MT}, and these works, in turn, motivated an increasing number of studies in the field catching the attention of researchers from Applied Mathematics to Physics and Computer Sciences. Our focus is on the Maki-Thompson model. This model was initially formulated in \cite{MT} in the context of homogeneously mixed populations. It is assumed that the population is represented by a complete graph with $n$ vertices, that is, all vertices are connected between them. Vertices represent individuals, and edges represent possible interactions. Individuals are classified into three classes: ignorants, spreaders, and stiflers. Ignorants are those individuals who do not know the rumor, spreaders are individuals who are propagating the rumor throughout the population, and stiflers are those individuals who know the information but are not participating in the propagation process. With these assumptions, the Maki-Thompson model, mathematically, can be defined as a continuous-time Markov chain $\{(X(t),Y(t))\}_{t\geq 0}$, with the following transitions and rates.

\begin{equation*}
\begin{array}{ccc}
\text{interaction} \qquad &\text{transition} \qquad &\text{rate}\\[0.2cm]
\text{spreader -- ignorant} \qquad &(-1, 1) \qquad &X Y,\\[0.2cm]
\text{spreader -- spreader/stifler} \qquad &(0, -1) \qquad &Y (n - X).
\end{array}
\end{equation*}
In this case, note that if the process is in state $(i,j)$ at time $t$, then the probabilities that it jumps to states $(i-1,j+1)$ or $(i,j-1)$ at time $t+h$ are, respectively, $ijh + o(h)$ and $j (n - i)h + o(h)$, where $o(h)$ represents a function such that $\lim_{h\to 0}o(h)/h =0$. The random variables $X(t)$ and $Y(t)$ represent the number of ignorants and spreaders, respectively, at time $t$, for $t\geq 0$. In addition, letting $Z(t)$ for the number of stiflers at time $t$, $t\geq 0$, we have $X(t)+Y(t)+Z(t)=n$, for all $t\geq 0$. Independent of the values for $X(0),Y(0),Z(0)$, it is not difficult to see that the process is absorbed at some point, and we can define its absorption time as $\tau_n:=\inf \{t\geq 0: Y(t)=0\}$. The first rigorous results for this model are limit theorems for the remaining proportion of ignorants at the end of the process. That is, the first results refer to the asymptotic behavior of $X(\tau_n)/n$ as $n\to \infty$. For a deeper discussion of limit theorems for these quantities for the Maki-Thompson model and some generalizations, always with the assumption of a complete graph, we refer the reader to \cite{LTRM} for a first reading and also to \cite{GR,RPRS,MRR} for further generalizations. 

The Maki-Thompson rumor model may also be defined on a graph, as a continuous-time Markov process $\{\eta_t\}_{t\geq 0}$ with states space $\mathcal{S}=\{0,1,2\}^{V}$, where $V$ denotes the set of vertices. In this case, at time $t$ the state of the process is a function $\eta_t: V \longrightarrow \{0,1,2\}$. Given a configuration $\eta \in \mathcal{S}$, we assume that each vertex $v\in V$ represents an individual, which is said to be, according to $\eta$, an ignorant if $\eta(v)=0,$ a spreader if $\eta(v)=1$, and a stifler if $\eta(v)=2.$ Then, if the system is in configuration $\eta,$ the state of vertex $v$ changes according to the following transition rates:

\begin{equation}\label{rates}
\begin{array}{rclc}
&\text{transition} &&\text{rate} \\[0.1cm]
0 & \rightarrow & 1, & \hspace{.5cm}  n_{1}(v,\eta),\\[.2cm]
1 & \rightarrow & 2, & \hspace{.5cm}   n_{1}(v,\eta) + n_{2}(v,\eta),
\end{array}
\end{equation}

\noindent where $$n_i(v,\eta)= \sum_{u\sim v} 1\{\eta(u)=i\},$$ 
is the number of neighbors of $v$ in state $i$ for the configuration $\eta$, for $i\in\{1,2\}.$ Formally, \eqref{rates} means that if the vertex $v$ is in state, say, $0$ at time $t$ then the probability that it will be in state $1$ at time $t+h$, for $h$ small, is $n_{1}(v,\eta) h + o(h)$ (see Fig. \ref{fig:realization}). For a review of recent rigorous results for this model on different graphs, we refer the reader to the following references: \cite{EAP,junior-rodriguez-speroto/2020,junior-rodriguez-speroto/2021,Diaz-rodriguez-rua/2025},  as well as the references cited therein. The primary rigorous results in the existing literature are related to the asymptotic behavior of the proportion of ignorants at the end of the process, in the context of finite graphs, or the propagation or non-propagation of the rumor in a specific sense, in the case of graphs with an infinite number of vertices.

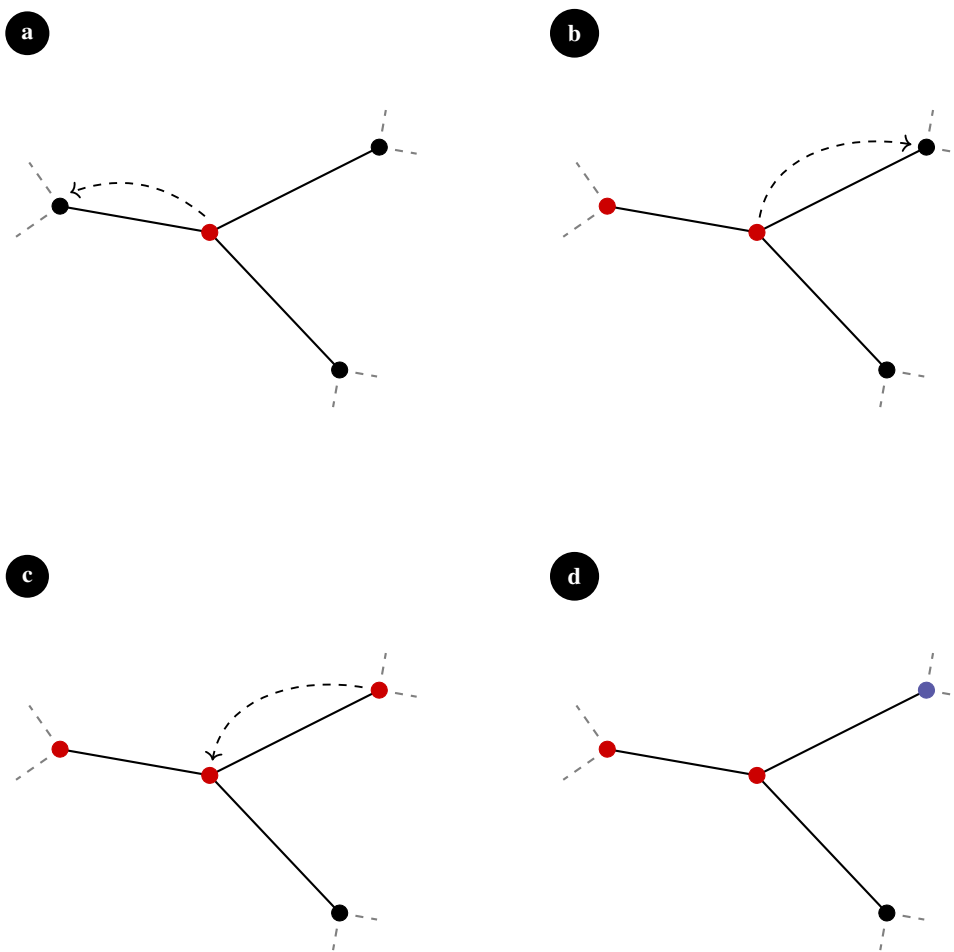
\begin{figure}[h!]
    \centering
\begin{tikzpicture}[scale=1.2, every node/.style={scale=1}]

\tikzset{every state/.style={minimum size=0pt,fill=black,draw=none,text=white}}




\node[state] at (-2,2.2) {\bf \footnotesize a};

\node[rotate=350] at (0,0) {

\begin{tikzpicture}

\draw[thick] (0,0) -- (2,0);
\draw[thick] (2,0) -- (4,1.5);
\draw[thick] (2,0) -- (4,-1.5);
\draw[thick,gray,dashed] (4,1.5) -- (4,2);
\draw[thick,gray,dashed] (4,1.5) -- (4.5,1.5);
\draw[thick,gray,dashed] (4,-1.5) -- (4,-2);
\draw[thick,gray,dashed] (4,-1.5) -- (4.5,-1.5);
\draw[thick,gray,dashed] (0,0) -- (-0.5,0.5);
\draw[thick,gray,dashed] (0,0) -- (-0.5,-0.5);

\filldraw [black] (0,0) circle (3pt);

\draw [->,line width=0.25mm,dashed] (1.9,0.2) to [out=150, in=30]  (0.1,0.2);

\filldraw [red!80!black] (2,0) circle (3pt);
\filldraw [black] (4,1.5) circle (3pt);
\filldraw [black] (4,-1.5) circle (3pt);
\end{tikzpicture}};

\node[state] at (4,2.2) {\bf \footnotesize b};

\node[rotate=350] at (6,0) {

\begin{tikzpicture}

\draw[thick] (0,0) -- (2,0);
\draw[thick] (2,0) -- (4,1.5);
\draw[thick] (2,0) -- (4,-1.5);
\draw[thick,gray,dashed] (4,1.5) -- (4,2);
\draw[thick,gray,dashed] (4,1.5) -- (4.5,1.5);
\draw[thick,gray,dashed] (4,-1.5) -- (4,-2);
\draw[thick,gray,dashed] (4,-1.5) -- (4.5,-1.5);
\draw[thick,gray,dashed] (0,0) -- (-0.5,0.5);
\draw[thick,gray,dashed] (0,0) -- (-0.5,-0.5);

\filldraw [red!80!black] (0,0) circle (3pt);
\filldraw [red!80!black] (2,0) circle (3pt);

\draw [->,line width=0.25mm,dashed] (2,0.2) to [out=90, in=180]  (3.8,1.5);

\filldraw [black] (4,1.5) circle (3pt);
\filldraw [black] (4,-1.5) circle (3pt);
\end{tikzpicture}};

\node[state] at (-2,-3.8) {\bf \footnotesize c};

\node[rotate=350] at (0,-6) {

\begin{tikzpicture}

\draw[thick] (0,0) -- (2,0);
\draw[thick] (2,0) -- (4,1.5);
\draw[thick] (2,0) -- (4,-1.5);
\draw[thick,gray,dashed] (4,1.5) -- (4,2);
\draw[thick,gray,dashed] (4,1.5) -- (4.5,1.5);
\draw[thick,gray,dashed] (4,-1.5) -- (4,-2);
\draw[thick,gray,dashed] (4,-1.5) -- (4.5,-1.5);
\draw[thick,gray,dashed] (0,0) -- (-0.5,0.5);
\draw[thick,gray,dashed] (0,0) -- (-0.5,-0.5);

\filldraw [red!80!black] (0,0) circle (3pt);

\draw [<-,line width=0.25mm,dashed] (2,0.2) to [out=90, in=180]  (3.8,1.5);

\filldraw [red!80!black] (2,0) circle (3pt);
\filldraw [red!80!black] (4,1.5) circle (3pt);
\filldraw [black] (4,-1.5) circle (3pt);
\end{tikzpicture}};

\node[state] at (4,-3.8) {\bf \footnotesize d};

\node[rotate=350] at (6,-6) {

\begin{tikzpicture}

\draw[thick] (0,0) -- (2,0);
\draw[thick] (2,0) -- (4,1.5);
\draw[thick] (2,0) -- (4,-1.5);
\draw[thick,gray,dashed] (4,1.5) -- (4,2);
\draw[thick,gray,dashed] (4,1.5) -- (4.5,1.5);
\draw[thick,gray,dashed] (4,-1.5) -- (4,-2);
\draw[thick,gray,dashed] (4,-1.5) -- (4.5,-1.5);
\draw[thick,gray,dashed] (0,0) -- (-0.5,0.5);
\draw[thick,gray,dashed] (0,0) -- (-0.5,-0.5);

\filldraw [red!80!black] (0,0) circle (3pt);

\draw [->,line width=0.4mm,white] (0.1,0.2) to [out=30, in=150]  (1.9,0.2);

\filldraw [red!80!black] (2,0) circle (3pt);
\filldraw [blue!30!gray] (4,1.5) circle (3pt);
\filldraw [black] (4,-1.5) circle (3pt);
\end{tikzpicture}};

\end{tikzpicture}

    \caption{Possible realization of the MT-model on a tree $\mathbb{T}$. The vertices of the tree represent individuals, each belonging to one of three categories: ignorants (black vertices), spreaders (red vertices), or stiflers (blue vertices). (a) A spreader passes the rumor to any of its nearest ignorant neighbors at rate one. (b) After receiving the rumor, the contacted ignorant becomes a spreader and begins spreading the information.
(c)-(d) At the same rate, a spreader turns into a stifler after coming into contact with neighboring spreaders or stiflers.}
    \label{fig:realization}
\end{figure}



In this work, we extend the approach of \cite{junior-rodriguez-speroto/2020,junior-rodriguez-speroto/2021} to the study of stochastic rumors on trees and random trees. In \cite{junior-rodriguez-speroto/2020}, the authors study the generalization known as the Maki–Thompson model with $k$-stifling on infinite Cayley trees $\mathbb{T}_d$, for $d\geq 2$. These are infinite deterministic trees in which every vertex has degree $d+1$. In contrast, \cite{junior-rodriguez-speroto/2021} considers stochastic rumors in random trees, which can be viewed as family trees generated by a branching process. In particular, one starts with a single vertex, the root of the tree, which produces offspring according to a given discrete distribution. Each of these offspring (if any) then independently generates further descendants according to the same law, and this continues indefinitely or until extinction at some generation. Such trees are known as Galton–Watson trees, or simply random trees. In both settings, the authors use branching process theory to analyze the existence of critical thresholds that determine whether the rumor dies almost surely or survives with positive probability. These results are particularly relevant in view of the fact that many random complex network models, which are more suitable for representing real populations, exhibit a local structure that is tree-like.

First, we define an extension of the Maki-Thompson rumor model on an infinite Cayley tree $\mathbb{T}_d$ by assuming that as soon as an individual hears the rumor, that individual either spreads it with probability $p\in (0,1]$, or stays neutral, becoming a stifler with probability $1-p$. For this model, we prove a phase transition result on $p$, and localize the critical threshold. Moreover, we study a second extension considering the Maki-Thompson model on random trees with three types of vertices: {\it hubs}, each connected to $d+1$ other vertices; {\it regular vertices}, each connected to $1<k=o(d)$ other vertices; and leaves, which are vertices with one neighbor only. Additionally, assume that each hub is connected, on average, to $\alpha (d+1)$ other hubs via paths of length $h$, where $\alpha\in (0,1]$. Under this assumption, we obtain a phase transition in $\alpha$ that depends on $d,k,$ and $h$. We illustrate our results for the specific case where $k$ is on the order of $\log d$. In the field of network theory, a hub is defined as a vertex with a significantly larger number of neighbors compared to other vertices within the network. The purpose of this extension is to gain an understanding, from a theoretical point of view, of the impact of the distance between hubs on the dissemination of a rumor in a network. In particular, this extension serves as a toy model for the representation of rumor spreading in Barabási-Albert like networks, which are mainly composed by trees formed by hubs, vertices with small connections and leaves. We refer the reader to \cite{Barabasi1999,Cohen2003,Albert1999} for a review of the properties and applications of such networks.  

The paper is divided into two parts. In Section 2, we present the basic notation and definitions, together with a formal description of the models and the main results. { It should be noted that our results are new and have not been addressed in the existing literature. In fact, our work extends previous studies within a broader framework, motivated by potential applications.} Section 3 is devoted to the proofs.

\section{The model and main results}

\subsection{A brief comment on notation, auxiliary identities, and approximations}

\subsubsection{Basic notation of Graph Theory}

During this work, we assume that the population is represented by an infinite tree $\mathbb{T}=(\mathcal{V},\mathcal{E})$. As usual, $ \mathcal{V} $ stands for the set of vertices and $\mathcal{E} \subset \{\{u,v\}: u,v \in \mathcal{V}, u \neq v\}$ stands for the set of edges. We shall abuse notation by writing $\mathcal{V}=\mathbb{T}$. We consider rooted trees identifying one vertex as the root of the tree and denoting it by ${\bf 0}$. If $\{u,v\}\in \mathcal{E}$, we say that $u$ and $v$ are neighbors, and we denote it by $u\sim v$. The degree of a vertex $v$, denoted by $deg(v)$, is the number of its neighbors. A path in $\mathbb{T}$ is a finite sequence $v_0, v_1, \dots, v_n $ of distinct vertices such that $ v_i \sim v_{i+1} $ for each $i$. For any tree, there is a unique path connecting any pair of distinct vertices $u$ and $v$ so we define the distance between them, denoted by $d(u,v)$, as the number of edges in that path. We denote by $\mathbb{T}_d$ the infinite Cayley tree of coordination number $d+1$, where $d\geq 2$. This is a graph with an infinite number of vertices, without cycles and such that every vertex has degree $d+1$. For each $v\in \mathcal{V}$ define $|v|:=d({\bf{0}},v)$. We denote by $\partial \mathbb{T}_{n}$ the set of vertices at distance $n$ from the root. That is, $\partial \mathbb{T}_{n}:= \{v \in \mathbb{T}: |v|=n\}$.


\subsubsection{Asymptotic notation and the incomplete gamma function}

Some of our results are asymptotic in nature, that is, we will assume that $d\rightarrow \infty$. Given two functions $f=f(n)$ and $g=g(n)$ we will write $f(n)=o(g(n))$ if $\lim_{n\to \infty} f(n)/g(n)=0$, and $f(n)=O(g(n))$ if $|f(n)|\leq M |g(n)|$, for all $n\geq n_0$, where $M$ is a positive constant and $n_0$ is a real number. We write $f(n)=\Theta(g(n))$ if $f(n)=O(g(n))$ and $g(n)=O(f(n))$. In addition, we write $f\sim g$ if $f(n)=(1+o(1))g(n)$, that is, $\lim_{n\to \infty} f(n)/g(n)=1$, and we write $f(n)\lesssim g(n)$ if there exists a function $h=h(n)$ such that $f(n)\leq h(n)$ and $h\sim g$. Some of our results are stated using the incomplete gamma function, which is defined by 
 \begin{equation}
 \label{eq:Gamanat}
     \Gamma(m, n) = (m - 1)! \, e^{-n} \sum_{i=0}^{m - 1} \frac{n^i}{i!}, \,\, m,n\in \mathbb{N}.
 \end{equation}  
The incomplete gamma function is well-defined for $n\in \mathbb{R}_+ \cup \{0\}$, but considering $n\in\mathbb{N}$ is sufficient for our purposes. To simplify certain steps in our proofs, we use the following relations.

 \begin{lemma}\label{lem:gamma}
     Consider the incomplete gamma function $\Gamma(m, n)$, $m,n\in\mathbb{N}$. Then,
     \begin{enumerate}
         \item[(i)] $\Gamma(m+1,n)=m\,\Gamma(m,n)+ n^m\, e^{-n}$.  
    \item[(ii)] $\Gamma(m,m+1) \sim \left(m/e\right)^m \sqrt{\pi/(2m)}.$
     \end{enumerate}
 \end{lemma}

\begin{proof}
    For (i) see \cite[Theorem 1]{jameson}. To prove (ii) note that by \eqref{eq:Gamanat} 
$$\Gamma(m,m+1)=(m - 1)! \, e^{-(m+1)} \sum_{i=0}^{m - 1} \frac{(m+1)^i}{i!},$$
and we can use Stirling's approximation $m!\sim m^{m+1/2}e^{-m}\sqrt{2\pi}$, and  
\begin{equation}\label{eq:ross}
\sum_{i=0}^{m}\frac{(m+1)^i}{i!}\sim \frac{e^{m+1}}{2},
\end{equation}
see \cite[page 146]{ross}, to conclude
$$
\begin{array}{rcll}
  \Gamma(m,m+1)   & = & \displaystyle(m - 1)! \, e^{-(m+1)} \left\{\sum_{i=0}^{m} \frac{(m+1)^i}{i!} - \frac{(m+1)^m}{m!}\right\} &\,\, \text{ by }\eqref{eq:Gamanat}\\[.4cm]
    &\sim &\displaystyle (m-1)!\, e^{-(m+1)}\left\{\frac{e^{m+1}}{2} -\frac{(m+1)^m}{m!}\right\}&\,\,\text{ by }\eqref{eq:ross}\\[.4cm]
    &= &\displaystyle  \frac{(m-1)!}{2} -\frac{(m+1)^m\,e^{-(m+1)}}{m}&\\[.4cm]
    &\sim &\displaystyle  \frac{e^{-(m+1)}}{m} \left\{e \, m^{m+1/2}\,\sqrt{\frac{\pi}{2}}-(m+1)^m\right\}&\,\,\text{ by Stirling}\\[.4cm]
        &\sim &\displaystyle  \left(\frac{m}{e}\right)^m \left\{\sqrt{\frac{\pi}{2m}}-\frac{1}{m}\right\}.&\\[.4cm]
\end{array}
$$
\end{proof}

\subsection{The Maki-Thompson rumor model on trees with probability $p$ of spreads}


We define an extension of the Maki-Thompson rumor model on an infinite tree $\mathbb{T}$ by assuming that as soon as an individual hears the rumor, that individual either spreads it with probability $p\in (0,1]$, or stays neutral, becoming a stifler, with probability $1-p$. The model is a continuous-time Markov process $(\eta_t)_{t\geq 0}$ with states space $\mathcal{S}=\{0,1,2\}^{\mathbb{T}}$. That is, at time $t$ the state of the process is a function $\eta_t: \mathbb{T} \longrightarrow \{0,1,2\}$. We assume that each vertex $v \in \mathbb{T}$ represents an individual, and we say that such an individual is, at time $t$, an ignorant if $\eta_t(v)=0,$ a spreader if $\eta_t(v)=1$, or a stifler if $\eta_t(v)=2.$ Moreover, if the system is in configuration $\eta \in \mathcal{S},$ the state of vertex $v$ changes according to the following transition rates

\begin{equation}\label{rates2}
\begin{array}{rclc}
&\text{transition} &&\text{rate} \\[0.1cm]
0 & \rightarrow & 1, & \hspace{.5cm}  p\,n_{1}(v,\eta),\\[.2cm]
0 & \rightarrow & 2, & \hspace{.5cm}  (1-p)\,n_{1}(v,\eta),\\[.2cm]
1 & \rightarrow & 2, & \hspace{.5cm}   n_{1}(v,\eta) + n_{2}(v,\eta),
\end{array}
\end{equation}

\noindent where $$n_i(v,\eta)= \sum_{u\sim v} 1\{\eta(u)=i\},$$ 
is the number of nearest neighbors of vertex $v$ in state $i$ for the configuration $\eta$, for $i\in\{1,2\}.$ Formally, \eqref{rates2} means that if the vertex $v$ is in state, say, $0$ at time $t$ then the probability that it will be in state $1$ at time $t+h$, for $h$ small, is $p\,n_{1}(x,\eta) h + o(h)$. Note that the rates in \eqref{rates2} represent how the changes of states of individuals depend on the states of its neighbors. While the change of state of an ignorant is influenced by its spreader neighbors, the change of state for a spreader is influenced by the number of non-ignorant neighbors. Note that stiflers do not interact with ignorants. Moreover, we point out that by letting $p=1$ we recover the basic Maki-Thompson rumor model on Cayley trees studied by \cite{junior-rodriguez-speroto/2020}. We call the Markov process $(\eta_t)_{t\geq 0}$ the Maki-Thompson rumor model on $\mathbb{T}$ with probability $p$ of spreads, MT$(\mathbb{T},p)$-model for short. In addition, we refer to the case when $\eta_0({\bf{0}})=1$ and $\eta_0(v)=0$ for all $v\neq {\bf 0}$ as the \textit{standard initial configuration}.     

\smallskip
\begin{definition}\label{defn:survival}
Let $p\in (0,1]$ and consider the MT$(\mathbb{T},p)$-model. We say that there is survival of the rumor if for any $t\geq 0$ there exist $v\in \mathbb{T}$ such that $\eta_t(v)=1$. Other case, we say that the rumor becomes extinct.
\end{definition}

Now, we focus our attention for the infinite Cayley tree of coordination number $d+1$, with $d\geq 2$, $\mathbb{T}=\mathbb{T}_d$. We denote the rumor survival probability as $\theta(d,p)$ and we observe that Definition \ref{defn:survival} is equivalent to \cite[Definition 1]{junior-rodriguez-speroto/2020}. By a coupling argument it is possible to prove, see Lemma \ref{lem:monot-MTp} in Section 3, that $\theta(d,p)$ is a non-decreasing function of $p$. Therefore we can define 
\begin{equation}\label{eq:defpc}
    p_c(d):=\inf\{p:\theta(d,p)>0\}.
\end{equation}
Note that $p_c(d)$ is a critical value of $p$ such that

$$\theta(d,p)\left\{
\begin{array}{cc}
 =0    &\text{ if }p<p_c(d),\\
>0    &\text{ if }p>p_c(d).
\end{array}
\right.
$$

Thus, the fact that $p_c(d)\in(0,1)$ -- i.e., $p_c(d)$ is non-trivial -- guarantees the existence of a phase transition in the behavior of the process. 

\begin{theorem}\label{thm:MTp}
Let $p\in (0,1], d\geq 3,$ and consider the MT$(\mathbb{T}_d,p)$-model with the standard initial configuration. Then 
    \begin{equation}\label{eq:pcdef}
    p_c(d)=\left\{ \frac{d e^{d+1}}{(d+1)^d} \Gamma(d, d+1)\right\}^{-1},
        \end{equation}
    where $\Gamma(d,d+1)$ is the incomplete gamma function defined by \eqref{eq:Gamanat}  (see Table \ref{tab:critico-GMT}). Moreover, $p_c(d)\in(0,1)$ for any $d\geq 3$, and 
 \begin{equation}\label{eq:pcasymp}
    p_c(d) \sim  \sqrt{\dfrac{2}{\pi d}}.
    \end{equation}
\end{theorem}

\begin{table}[h!]
    \centering
\begin{tabular}{cccccccccc}
\hline
$d$ & $3$  & $4$  & $5$  & $6$  & $7$  & $8$  & $9$  & $10$ & $11$ 
\\ \hline
$p_c$ &  $0.8205$  &  $0.6620$  &  $0.5634$  &  $0.4955$  &  $0.4454$ &  $0.4067$  &  $0.3759$ &  $0.3505$ & $0.3293$\\\hline
  \end{tabular} 
\caption{Values of $p_c(d)$ for $d\in\{3,\ldots 11\}$.}
    \label{tab:critico-GMT}
\end{table}


\begin{corollary}
Let $d\geq 3$ and consider the MT$(\mathbb{T}_d,1)$-model with the standard initial configuration. Then, $\theta(d,1)>0$. 
\end{corollary}

\begin{theorem}
\label{T:234}
Consider the MT$(\mathbb{T}_d, p)$-model with the standard initial configuration. Then,
 \begin{equation*}
\theta(d,p)=1-\frac{1}{d+1} \displaystyle\sum_{i=0}^{d+1} \left( \dfrac{p\psi}{1-p} \right)^{i} \,  \sum_{k = i}^{d+1} k\,k!\binom{k}{i} \binom{d+1}{k} \left( \dfrac{1-p}{d+1} \right)^{k},
\end{equation*}
 when $\psi$ is the smallest non-negative root of the equation 
 $$\frac{d}{d+1}\, \left(\frac{sp + 1 - p}{d + 1}\right)^{d-1} e^{\frac{d + 1}{sp + 1 - p}}\, \Gamma\left(d, \frac{d + 1}{sp + 1 - p} \right) p(s-1) + 1 = s. $$

\end{theorem}

The key to proving Theorems \ref{thm:MTp} and \ref{T:234} is to look at the original model as a branching process, so that we can apply well-known results from the Theory of Branching Processes (see \cite[Chapter 2]{schinazi}). This approach allows us to study the model on random trees. See, for example, \cite{junior-rodriguez-speroto/2021} where the authors study the Maki-Thompson model on Galton-Watson trees. In the next section, we explore this idea by studying the Maki-Thompson model on a class of random inhomogeneous trees. These trees are formed by hubs of degree $d+1$, some of which are connected by paths of length $h$, and we assume that all other vertices have degree at most $k=o(d)$, with a special focus in the case $k=k(d) = \Theta( \log d)$.

\subsection{The Maki-Thompson model in a class of inhomogeneous trees}

We consider trees with three types of vertices: {\it hubs}; each connected to $d+1$ other vertices; {\it regular vertices}, each connected to $1<k=o(d)$ other vertices; and leaves, which are vertices with one neighbor only. Then the tree is generated randomly as follows. Assume that the root is a hub and that each of its neighbors is, with probability $\alpha$, a regular vertex that is connected through a path to another hub, or, with probability $1 - \alpha$, is a leaf. Furthermore, if a neighbor connects to another hub via a path, the distance between these hubs is $h \in \mathbb{N}$, and all the vertices along the path are regular vertices, connected either to leaves or to other regular vertices on the same path. See Figure~\ref{fig:special-tree}. We repeat this construction for each new hub, and no additional connections are made to the regular vertices created during the process.  We denote this random tree by $\mathbb{T}_{d,k,\alpha,h}$, and note that for $d \geq 2$, if $\alpha = 1$ and $h = 1$, then this construction yields the infinite Cayley tree $\mathbb{T}_d$.

\begin{figure}
\begin{center}
\begin{tikzpicture}[scale=1.6, every node/.style={scale=1.4}]

\tikzset{every state/.style={minimum size=0pt,fill=black,draw=none,text=white}}

\node[state] at (-1.5,1) {\bf \footnotesize a};
\node[state] at (5.6,1) {\bf \footnotesize b};
\node[state] at (-1.5,-1.8) {\bf \footnotesize c};


\draw[color=red, fill=red!10, dashed] (4.6,-0.6) rectangle (5.5,0.2);
\draw[color=red, fill=red!10, dashed] (1.2,-4.2) rectangle (5.2,-1.8);

\draw [->,line width=0.6mm,red!80!black] (5.05,-0.6) to [out=270, in=90]  (3.2,-1.8);

\node at (0,-0.2) {
\begin{tikzpicture}
\filldraw [black] (0,0) circle (1.5pt); 
\filldraw [black] (0.5,0.75) circle (1.5pt); 
\filldraw [black] (-0.5,0.75) circle (1.5pt); 
\filldraw [black] (1,0) circle (1.5pt); 
\filldraw [black] (-1,0) circle (1.5pt);
\filldraw [black] (0.5,-0.75) circle (1.5pt);
\filldraw [black] (-0.5,-0.75) circle (1.5pt);

\node at (0,-0.3) {\tiny\bf $0$};

\draw (0,0) -- (0.5,0.75); 
\draw (0,0) -- (-0.5,0.75); 
\draw (0,0) -- (1,0); 
\draw (0,0) -- (-1,0);
\draw (0,0) -- (0.5,-0.75);
\draw (0,0) -- (-0.5,-0.75);
\end{tikzpicture}};

\node at (4,0) {
\begin{tikzpicture}
\filldraw [black] (0,0) circle (1.5pt); 
\filldraw [black] (0.5,0.75) circle (1.5pt); 
\filldraw [black] (-0.5,0.75) circle (1.5pt); 
\filldraw [black] (1,0) circle (1.5pt); 
\filldraw [black] (-1,0) circle (1.5pt);
\filldraw [black] (0.5,-0.75) circle (1.5pt);
\filldraw [black] (-0.5,-0.75) circle (1.5pt);

\node at (0,-0.3) {\tiny\bf $0$};
\node at (-1,-0.25) {\tiny\bf $u$};
\node at (0.5,0.5) {\tiny\bf $v$};
\node at (1,-0.25) {\tiny\bf $w$};

\draw [orange,very thick] (1,0) circle (3pt);
\node at (1.4,0) {$\dots$};     
\draw [orange,very thick] (-1,0) circle (3pt);
\node at (-1.4,0) {$\dots$};     
\draw [orange,very thick] (0.5,0.75) circle (3pt);
\node at (0.8,1) {$\adots$};

\draw (0,0) -- (0.5,0.75); 
\draw (0,0) -- (-0.5,0.75); 
\draw (0,0) -- (1,0); 
\draw (0,0) -- (-1,0);
\draw (0,0) -- (0.5,-0.75);
\draw (0,0) -- (-0.5,-0.75);
\end{tikzpicture}};

\node at (2,-3) {
\begin{tikzpicture}
\filldraw [black] (0,0) circle (1.5pt); 
\filldraw [black] (0.5,0.75) circle (1.5pt); 
\filldraw [black] (-0.5,0.75) circle (1.5pt); 
\filldraw [black] (1,0) circle (1.5pt); 
\filldraw [black] (-1,0) circle (1.5pt);
\filldraw [black] (0.5,-0.75) circle (1.5pt);
\filldraw [black] (-0.5,-0.75) circle (1.5pt);
\filldraw [black] (2,0) circle (1.5pt);

\filldraw [black] (4,0) circle (1.5pt); 
\filldraw [black] (4.5,0.75) circle (1.5pt); 
\filldraw [black] (3.5,0.75) circle (1.5pt); 
\filldraw [black] (3,0) circle (1.5pt);
\filldraw [black] (4.5,-0.75) circle (1.5pt);
\filldraw [black] (3.5,-0.75) circle (1.5pt);
\filldraw [black] (5,0) circle (1.5pt); 

\filldraw [black] (1,0.4) circle (1.5pt); 
\filldraw [black] (1,-0.4) circle (1.5pt); 
\filldraw [black] (2,0.4) circle (1.5pt); 
\filldraw [black] (2,-0.4) circle (1.5pt); 
\filldraw [black] (3,0.4) circle (1.5pt);
\filldraw [black] (3,-0.4) circle (1.5pt); 

\node at (0,-0.3) {\tiny\bf $0$};
\node at (-1,-0.25) {\tiny\bf $u$};
\node at (0.5,0.5) {\tiny\bf $v$};
\node at (1.2,-0.25) {\tiny\bf $w$};
\node at (4.5,0.5) {\tiny\bf $x$};
\node at (4.5,-1) {\tiny\bf $y$};

\draw (1,-0.4) -- (1,0.4); 
\draw (2,-0.4) -- (2,0.4); 
\draw (3,-0.4) -- (3,0.4); 

\draw (4,0) -- (4.5,0.75); 
\draw (4,0) -- (3.5,0.75); 
\draw (4,0) -- (5,0); 
\draw (4,0) -- (4.5,-0.75);
\draw (4,0) -- (3.5,-0.75);

\draw [orange,very thick] (1,0) circle (3pt);
\draw [orange,very thick] (-1,0) circle (3pt);
\node at (-1.4,0) {$\dots$};     
\draw [orange,very thick] (0.5,0.75) circle (3pt);
\node at (0.8,1) {$\adots$};

\draw (0,0) -- (0.5,0.75); 
\draw (0,0) -- (-0.5,0.75); 
\draw (0,0) -- (1,0) -- (2,0) -- (3,0) -- (4,0); 
\draw (0,0) -- (-1,0);
\draw (0,0) -- (0.5,-0.75);
\draw (0,0) -- (-0.5,-0.75);

\draw [orange, very thick] (4.5,-0.75) circle (3pt);
\node at (4.8,-1) {$\ddots$};
\draw [orange, very thick] (4.5,0.75) circle (3pt);
\node at (4.8,1) {$\adots$};

\end{tikzpicture}};
\end{tikzpicture}
\end{center}
\caption{Illustration of the first steps in the generation of $\mathbb{T}_{d,k,\alpha,h}$. For the sake of simplicity we consider $d=5, k=4$, and $h=4$. (a) The tree starts with a hub of degree $d+1$, which is chosen as the root and denoted by $\mathbf{0}$. (b) Each neighbor of the root is, with probability $\alpha$, a regular vertex connected by a path to another hub, or, with probability $1-\alpha$, a leaf. In this example, only the vertices $u, v,$ and $w$ are assumed to connect to another hubs. (c) We then reveal the connections of those vertices that link to other hubs. Here we show the path associated with $w$, which connects it to another hub. This new hub, in turn, may also have neighbors connected to other hubs, and the process continues generating the inhomogeneous tree $\mathbb{T}_{d,k,\alpha,h}$.}\label{fig:special-tree}
\end{figure}
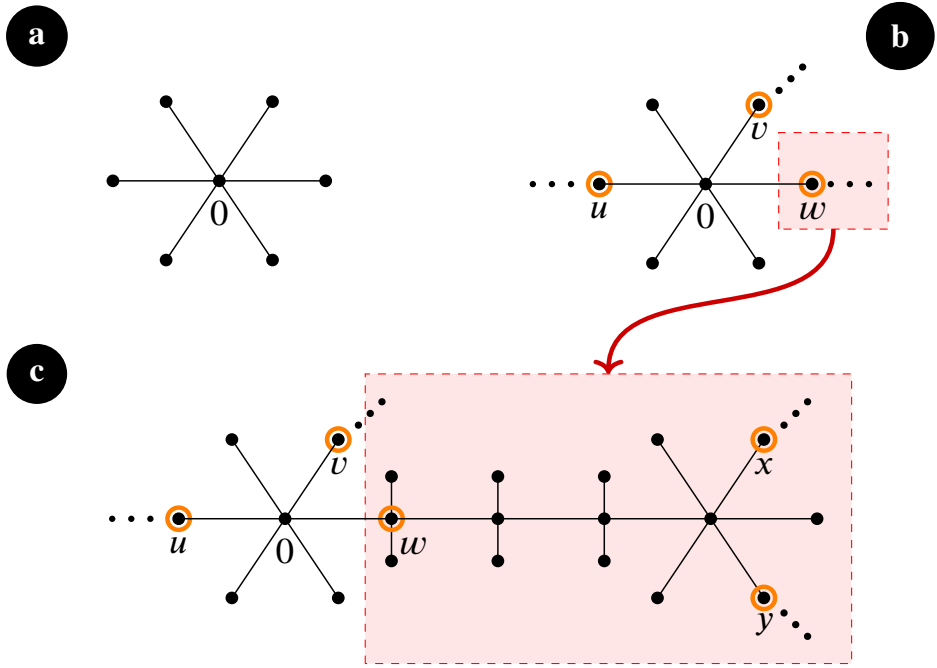

We consider the Maki--Thompson model on $\mathbb{T}_{d,k,\alpha,h}$, starting from the standard initial configuration, and our focus will be on how the rumor spreads through the hubs. For the sake of simplicity we denote the process by MT$(\mathbb{T},d,k,\alpha,h)$-model. Since we are interested in defining the model on an infinite tree, we condition on the event that the tree does not go extinct with positive probability. This occurs when $\alpha > 1/(d+1)$.

\smallskip
\begin{definition}
Let $d\geq 2$, $\alpha \in (0,1]$, $h\in\mathbb{N}$ and consider the MT$(\mathbb{T},d,k,\alpha,h)$-model. We say that there is survival of the rumor if for any $t\geq 0$ there exist $v\in \mathbb{T}_{d,k,\alpha,h}$ such that $\eta_t(v)=1$. Other case, we say that the rumor becomes extinct.
\end{definition}

We denote the rumor survival probability as $\theta(d,k,\alpha,h)$ and we observe that  by a coupling argument, it is possible to prove that $\theta(d,k,\alpha,h)$ is a non-decreasing function of $\alpha$, so we can define $\alpha_c(d,k,h):=\inf\{\alpha:\theta(d,k,\alpha,h)>0\}.$ Thus, we have that $\alpha_c(d,k,h)$ is a critical value of $\alpha$ such that

$$\theta(d,k,\alpha,h)\left\{
\begin{array}{cc}
 =0    &\text{ if }\alpha<\alpha_c(d,k,h),\\
>0    &\text{ if }\alpha>\alpha_c(d,k,h).
\end{array}
\right.
$$

\begin{theorem}\label{theo:princ}
Let $\alpha\in (0,1], d\geq 3,$ $k<d$ and consider the MT$(\mathbb{T},d,k,\alpha,h)$-model with the standard initial configuration. Then 
\begin{equation}
    \alpha_c(d,k,h)=p_c(d)\left\{\frac{e^{k}\Gamma(k-1,k)-\Gamma(k-1)}{k^{k-1}}\right\}^{1-h},
\end{equation}
where $\Gamma(m):=(m-1)!$ for $m\in\mathbb{N}$, $\Gamma(m,n)$ is the incomplete gamma function defined in \eqref{eq:Gamanat}, and $p_c(d)$ is the critical probability for the MT$(\mathbb{T}_d,p)$-model given in \eqref{eq:pcdef}.
\end{theorem}

\begin{corollary}\label{cor:1special}
Let $\alpha\in (0,1], d\geq 3,$ $k=o(d)$ and consider the MT$(\mathbb{T},d,k,\alpha,h)$-model with the standard initial configuration. $\alpha_c(d,k,h)\in(0,1)$ if, and only if, $h \lesssim \log d / \log k$.
\end{corollary}

\begin{corollary}\label{cor:logd}
Let $\alpha\in (0,1], d\geq 3,$ $k=\Theta(\log\,d)$ and consider the MT$(\mathbb{T},d,k,\alpha,h)$-model with the standard initial configuration. $\alpha_c(d,k,h)\in(0,1)$ if, and only if,
 \begin{equation}
h \lesssim \Theta\left( \frac{\log d}{\log \log d}\right).
 \end{equation} 
\end{corollary}

\section{Proofs}

\subsection{Preliminary results}

For our analysis, it is important to have information about the distribution of the number of spreaders that a spreader generates. For the MT\((\mathbb{T}_d, 1)\)-model, the standard Maki--Thompson model on the infinite Cayley tree of coordination number $d+1$, with $\eta_0(\textbf{0}) = 1$ and $\eta_0(x) = 0$ for all $x \neq \textbf{0}$ the following result was proved by \cite{junior-rodriguez-speroto/2020}. 

\begin{lemma}\label{lem:AdaltoEX}
\cite[Lemma 2]{junior-rodriguez-speroto/2020} Consider the MT$(\mathbb{T}_d,1)$-model with the standard initial configuration, and let $X$ be the number of spreaders that a spreader, different of the root, generates. Then:
\begin{equation*}
\label{eq4}
\mathbb{P}(X = i) = \binom{d}{i} \frac{(i+1)!}{(d+1)^{i+1}}, \quad i \in \{0, 1,  \dots, d\}.
\end{equation*}
Moreover, $\mathbb{E}(X) > 1$ if and only if $d \geq 3$.
\end{lemma}

Although \cite{junior-rodriguez-speroto/2020} does not provide an explicit expression for $\mathbb{E}(X)$, it can be derived as a consequence of \cite[Proposition 2.1]{junior-rodriguez-speroto/2021}. For completeness, we include the derivation of $\mathbb{E}(X)$ below.

\begin{lemma}
\label{lemma:224} Consider the MT$(\mathbb{T}_d,1)$-model with the standard initial configuration. Then,
\begin{equation}\label{eq:meanX}
\mathbb{E}(X) = \frac{d e^{d+1}}{(d+1)^{d}}\Gamma(d,d+1).
\end{equation}
\end{lemma}

\begin{proof}
By \cite[Proposition 2.1]{junior-rodriguez-speroto/2021}, see in its proof specifically in the case where $\xi = d$ almost surely, is possible to obtain:
\[
\mathbb{E}(X) = \frac{e^{d+1}}{(d+1)^{d}}\Gamma(d+1,d+1) -1,
\]
but since $\Gamma(d+1,d+1)=d\,\Gamma(d,d+1)+ (d+1)^d\, e^{-(d+1)}$, see Lemma \ref{lem:gamma}(i), we conclude the result.

\end{proof}

Another quantity of interest, for our purposes, is the probability that a spreader, other than the root, contacts a specific nearest neighbor before becoming stifler.

\begin{lemma}\label{lem:beta}
Consider the MT$(\mathbb{T}_d,1)$-model with the standard initial configuration. Let $\beta(d)$ be the probability that a spreader, other than the root, contacts a specific nearest neighbor before becoming stifler. Then,
\begin{equation}
\beta(d)=\frac{e^{d+1}\Gamma(d,d+1)-\Gamma(d)}{(d+1)^{d}}\sim \sqrt{\frac{\pi}{2d}}.
\end{equation}
\end{lemma}

\begin{proof} Note that if we choose a vertex $u$ in the spreader state, other than the root, to analyze the probability $\beta(d)$, then as soon as $u$ becomes a spreader it has one neighbor already in the spreader state (the one who informed it) and $d$ neighbors in the ignorant state. Our goal is to compute the probability that $u$ contacts a fixed vertex, say $v$, before turning into a stifler. To do this, let us define $C_i$ as the event that the first contact with $v$ occurs on the $ith-$attempt, for $i\in\{1,\ldots,d\}$. Then,

$$\beta(d)=\mathbb{P}\left(\displaystyle\bigcup_{i=1}^d C_i\right)=\sum_{i=1}^d \mathbb{P}(C_i),$$
where 

$$\mathbb{P}(C_i)=\left(\frac{d-1}{d+1}\right) \left(\frac{d-2}{d+1}\right) \cdots \left(\frac{d-i}{d+1}\right) \frac{1}{d+1}.$$

Therefore,

\begin{align*}
    \beta(d) & = \dfrac{1}{d+1} +  \left(\dfrac{d-1}{d+1}\right) \dfrac{1}{d+1} +  \left(\dfrac{d-2} {d+1}\right) \left(\dfrac{d-1}{d+1}\right)  \dfrac{1}{d+1} + \cdots +  \prod_{j=2}^{d}  \left( \dfrac{d+1-j}{d+1} \right)  \\    
     & = (d-1)! \left(\dfrac{1}{d+1}\right)^{d-1}   + \cdots +  \dfrac{(d-1)!}{(d-3)!} \left(\dfrac{1}{d+1}\right)^{3}   + \dfrac{(d-1)!}{(d-2)!} \left(\dfrac{1}{d+1}\right)^{2}   + \dfrac{(d-1)!}{(d-1)!}\left(\dfrac{1}{d+1} \right)   \\ 
     & {\small= \frac{(d+1)!}{d}\left\{  \left(\dfrac{1}{d+1}\right)^{d} + \cdots +  \dfrac{1}{(d-3)!} \left(\dfrac{1}{d+1}\right)^{4}  + \dfrac{1}{(d-2)!} \left(\dfrac{1}{d+1}\right)^{3}  + \dfrac{1}{(d-1)!}\left(\dfrac{1}{d+1} \right)^2\right\}}  \\
     & = \frac{(d+1)!}{d}\left\{ \displaystyle\sum_{i=2}^{d} \dfrac{1}{(d+1-i)!}\left(\dfrac{1}{d+1}\right)^i \right\} \\
    & =   \dfrac{(d+1)!}{ d\,(d+1)^{(d+1)}}  \displaystyle\sum_{i=1}^{d-1} \dfrac{(d+1)^i}{i!}.
\end{align*}

 Thus, since by \eqref{eq:Gamanat} we have
 $$ \Gamma(d, d+1) = (d - 1)! \, e^{-(d+1)} \sum_{i=0}^{d - 1} \frac{(d+1)^i}{i!},$$
then we obtain
 $$\beta(d)= \frac{(d+1)!}{ d\,(d+1)^{(d+1)}} \left\{ \frac{e^{d+1}}{(d-1)!}\Gamma(d,d+1)-1\right\}=\frac{e^{d+1}\Gamma(d,d+1)-\Gamma(d)}{(d+1)^{d}}.$$

 The approximation for $\beta(d)$ results as a consequence of the last equality, Lemma \ref{lem:gamma}(ii), and Stirling. 
 


\end{proof}

So far, we have reviewed some results for the MT$(\mathbb{T}_d, 1)$-model. In this version, individuals who become aware of the rumor attempt to transmit it to their nearest neighbors with probability one. A natural extension is the MT$(\mathbb{T}_d, p)$-model, in which each individual aware of the rumor transmits it with probability $p \in (0,1)$ and refrains from transmitting it with probability $1 - p$. As defined, this model corresponds to the Maki–Thompson version of the misinformation spreading in a passive environment model considered by \cite{gomez-junior-rodriguez/2024}. In what follows, we present results analogous to those above for the MT$(\mathbb{T}_d, p)$-model.

\begin{lemma}\label{lem:EXl}
Consider the MT$(\mathbb{T}_d, p)$-model with the standard initial configuration, and let $X'$ be the number of spreaders that a spreader, different of the root, generates. Then, the probability generating function of $X'$ is given by:
\begin{align*}
G_{X'}(s) & = \frac{d}{d+1}\, \left(\frac{sp + 1 - p}{d + 1}\right)^{d-1} e^{\frac{d + 1}{sp + 1 - p}}\, \Gamma\left(d, \frac{d + 1}{sp + 1 - p} \right) p(s-1) + 1. 
\end{align*}
Moreover, $\mathbb{E}(X') = p \mathbb{E}(X)$, where $X$ is the number of individuals that a spreader, different from the root, contacts.
\end{lemma} 

\begin{proof} 
Let $X$ be the number of individuals contacted by a different spreader of the root before becoming stifler and note that, given $X=i$, with $i\in\{0,1,\ldots,d\}$, $X' \sim \textrm{Binomial}(i, p)$. That is, we can write  
\begin{equation}\label{eq:XX}
X' = \sum_{\ell=1}^{X} I_\ell,
\end{equation}
where the random variables $I_\ell$'s are i.i.d. with $I_\ell \sim \mathrm{Bernoulli}(p)$. Indeed, each $I_\ell$ can be interpreted as
$I_\ell = 1$ if the $\ell$-th contacted individual transmits the rumor, or $I_\ell =0$, otherwise, for $\ell \in \{1, 2, \ldots, X\}$. Then, the probability generating function of $X'$ is given by:
\begin{equation}
\label{eq:lou}
    G_{X'}(s) = \mathbb{E}(s^{X'})  = \sum_{n=0}^{d} \mathbb{E}\left(s^{X'} \mid X = n\right) \, \mathbb{P}(X = n).
\end{equation}

Using Lemma \ref{lem:AdaltoEX} and \eqref{eq:lou}, we obtain
\begin{align*}
G_{X'}(s) 
&= \sum_{n=0}^{d} \left(sp + 1 - p\right)^n  n! \binom{d}{n}  \frac{n+1}{(d+1)^{n+1}} \\
&= \frac{d!}{d+1} \sum_{n=0}^{d} \frac{n+1}{(d - n)!} \left(\frac{sp + 1 - p}{d+1}\right)^n.
\end{align*}

Let us denote 
$$\zeta:=\zeta (s) = \frac{sp + 1 - p}{d + 1}, \quad \text{where } \zeta > 0,
$$
and note that we can rewrite:

$$
G_{X'}(s) = \frac{d!}{d+1} \sum_{n=0}^{d} \frac{(n+1)\, \zeta^n}{(d - n)!}.
$$

We begin by applying the change of index $k = d  - n$, yielding:
\begin{align*}
    G_{X'}(s) &= \frac{d!}{d+1} \sum_{k=0}^{d} \frac{(d-k+1)\, \zeta^{d-k}}{k!} \\
    &= \dfrac{d!}{d+1} \cdot \zeta^{d} \left[ (d+1) \sum_{k=0}^{d} \frac{\zeta^{-k}}{k!} - \dfrac{1}{\zeta} \sum_{k=1}^{d} \frac{\, \zeta^{-(k-1)}}{(k-1)!} \right] \\
    &= \dfrac{d!}{d+1} \cdot \zeta^{d} \left[ \frac{(d+1)}{d!}e^{\frac{1}{\zeta}}\Gamma\left(d+1, \frac{1}{\zeta} \right)  -  \frac{e^{\frac{1}{\zeta}}}{\zeta(d-1)!}\Gamma\left(d, \frac{1}{\zeta} \right)  \right].
\end{align*}
Using Lemma \ref{lem:gamma}(i), and performing appropriate simplifications, we obtain that

$$G_{X'}(s) = \frac{d}{d+1}\, \left(\frac{sp + 1 - p}{d + 1}\right)^{d-1} e^{\frac{d + 1}{sp + 1 - p}}\, \Gamma\left(d, \frac{d + 1}{sp + 1 - p} \right) p(s-1) + 1.  $$

By \eqref{eq:XX} the Wald's equation, see \cite[Proposition 11.4]{ross}, yields $\mathbb{E}(X') = p \, \mathbb{E}(X)$.

\end{proof}


It is worth noting that the previous results for $X'$ are sufficient for our purposes. However, it is not difficult to prove similar results for the number of spreaders generated by the initial spreader; that is, the root. The following results hold in this case.

\begin{lemma}\label{lem:JuniorN}
\cite[Lemma 1]{junior-rodriguez-speroto/2020} Consider the MT$(\mathbb{T}_d,1)$-model with the standard initial configuration, and let $N$ be the number of spreaders generated by the initial spreader; i.e., by the root. Then:
\begin{equation*}
\label{eq3}
\mathbb{P}(N = i) = i! \binom{d+1}{i} \frac{i}{(d+1)^{i+1}}, \quad i \in \{1, 2, \dots, d+1\}.
\end{equation*}
\end{lemma}

Similarly to how we proved Lemma \ref{lem:EXl} using Lemma \ref{lem:AdaltoEX}, we can state the next lemma based on Lemma \ref{lem:JuniorN}. Since the proof is very similar to that of Lemma~\ref{lem:EXl}, we omit the most part and only include the deduction of the law of $N'$.

\begin{lemma}
\label{L5}
Consider the MT$(\mathbb{T}_d, p)$-model with the standard initial configuration, and let $N'$ be the number of spreaders generated by the initial spreader. Then, the law and the probability generating function of $N'$ are given, respectively, by

\begin{equation}\label{eq:lawN'}
    \mathbb{P}(N' = i) =  \left( \dfrac{p}{1-p} \right)^{i}  \dfrac{1}{d+1} \sum_{k = i}^{d+1} k\,k! \binom{k}{i} \binom{d+1}{k} \left( \dfrac{1-p}{d+1} \right)^{k},
\end{equation}
for $i\in\{0,\ldots,d+1\}$, and 

\begin{align*}
G_{N'}(s) &=d\, \left(\frac{sp + 1 - p}{d + 1}\right)^{d} e^{\frac{d + 1}{sp + 1 - p}}\, \Gamma\left(d, \frac{d + 1}{sp + 1 - p} \right) p(s-1) + (sp+1-p).
\end{align*}

Moreover, $\mathbb{E}(N') = p \mathbb{E}(N)$, where $N$ is the number of individuals contacted by the initial spreader.
\end{lemma}

\begin{proof}
Since the proof of the expression for the probability generating function is very similar to that of Lemma~\ref{lem:EXl}, we omit it. Note that, given $N=k$, $N' \sim \text{Binomial}(k, p)$. Thus, by Lemma \ref{lem:JuniorN} we obtain
\begin{align*}
    \mathbb{P}(N' = i) &= \sum_{k = i}^{d+1} \mathbb{P}(N' = i \mid N = k) \mathbb{P}(N = k) \\
    &= \sum_{k = i}^{d+1}  \binom{k}{i} p^{i} (1-p)^{k-i} k! \binom{d+1}{k} \frac{k}{(d+1)^{k+1}}  \\
    &=  \left( \dfrac{p}{1-p} \right)^{i}  \dfrac{1}{d+1} \sum_{k = i}^{d+1} k\,k! \binom{k}{i} \binom{d+1}{k} \left( \dfrac{1-p}{d+1} \right)^{k}.      
\end{align*}

Since the proof of the expression for the probability generating function is very similar to that of Lemma~\ref{lem:EXl}, we omit it.
\end{proof}

\begin{remark}
As noted by \cite{junior-rodriguez-speroto/2020}, Lemma \ref{lem:JuniorN} is interesting in its own right, due to its connection with the Coupon Collector’s Problem. See \cite{flajolet-gardy-thimonier/1992,isaac,motwani} and the references therein for some works related to this problem. In our context, the problem can be stated as follows: At each stage, a collector obtains a coupon that is equally likely to be any one of $d+1$ types. Assuming that the outcomes of successive stages are independent, one interesting question is: What is the expected number of coupons drawn before a duplicate appears--that is, a coupon that has already been collected? This expected value is exactly $\mathbb{E}(N)$. Moreover, if we assume that, at each stage, the collector has a probability $1-p$ of losing or discarding the collected coupon, then $\mathbb{E}(N')$ of Lemma \ref{L5} represents the expected number of coupons drawn and kept until encountering a duplicate.
\end{remark}

\subsection{An underlying branching process of the MT$(\mathbb{T}_d, p)$-model}

Our first task is to contruct the underlying branching process related to the MT$(\mathbb{T}_d, p)$-model. For a reference of the Theory of Branching Processes with applications in the modelling of biological phenomena we refer the reader to \cite[Chapter 2]{schinazi}. Given $n \geq 0$, remember that we denote the $n$-th level of $\mathbb{T}_d$ as $\partial \mathbb{T}_{d,n} := \{ v \in \mathbb{T}_d : |v| = n \}$. If $\mathcal{S}_n$ denote the set of vertices belonging to the $(n+1)$-th level of $\mathbb{T}_d$ that eventually became spreaders, then
$$
\mathcal{S}_n := \left\{ v \in \partial \mathbb{T}_{d,n+1} : \bigcup_{t>0} \{ \eta_t(v) = 1 \} \right\},
$$
for all $n \in \mathbb{N}$. Note that by definition, $\mathcal{S}_0$ consists of all vertices at distance one from the root $\textbf{0}$ that eventually became spreaders; $\mathcal{S}_1$ consists of all vertices at distance two from the root that eventually became spreaders; and so on. We define the random variable $Z_n := |\mathcal{S}_n|$; that is, $Z_n$ is the number of vertices belonging to the $(n+1)$-th level of $\mathbb{T}_d$ that eventually became spreaders. In this way, $Z_0$ has the same distribution as $N'$, and moreover, for all $n \in \mathbb{N} \cup \{0\}$, we have
\begin{equation}\label{eq:Zn}
Z_{n+1} = \sum_{i=1}^{Z_n} X'_i,
\end{equation}
where $X'_1, X'_2, \ldots$ are independent and identically distributed copies of the random variable $X'$. Thus defined, the sequence $(Z_n)_{n \geq 0}$ is a branching process with an initial number of particles distributed as $N'$, and offspring distribution given by $X'$. 

\begin{lemma}
\label{L:233}
The MT$(\mathbb{T}_d,p)$-model survives, if and only if, the branching process $(Z_n)_{n \geq 0}$ survives.
\end{lemma}

\begin{proof}
Note that the event $\{Z_n \geq 1\}$ is equivalent to the event that there exists a path connecting the root $\mathbf{0}$ to $\partial \mathbb{T}_{d,n+1}$, although which the rumor propagates. In other words, there is a sequence of vertices $\mathbf{0}=v_0, v_1, \dots, v_{n+1}$ and corresponding times $0=t_0 < t_1 < \dots < t_{n+1}$, such that $\eta_{t_j}(v_j) = 1$ for all $j \in \{0,1,2,\dots,n+1\}$. If, for $A\subset \mathbb{T}_d$, we denote by $u \overset{T}{\rightarrow} A$ the transmission of information from $u$ to $v$, for some $v\in A$, then the MT$(\mathbb{T}_d,p)$-model survives if, and only, if 
$$
\bigcap_{n \geq 1} \{ \mathbf{0} \overset{T}{\rightarrow} \partial \mathbb{T}_{d,n} \},
$$
occurs. But, give the above discussion, this event is equivalent to the event
$$
\bigcap_{n = 1}^\infty \{ Z_n \geq 1 \},
$$  
which is the event of survival of the branching process $(Z_n)_{n \geq 0}$.

\end{proof}





\subsection{Proof of the main theorems}

\subsubsection{Monotonicity of $\theta(d,p)$.}

The critical value of $p$, $p_c(d)$, is well-defined by \eqref{eq:defpc} due to the following result.

\begin{lemma}\label{lem:monot-MTp}
Consider the MT$(\mathbb{T}_d, p)$-model with the standard initial configuration. For any $d\geq 2$, $\theta(d,p)$ is non-decreasing as a function of $p$.
\end{lemma}

\begin{proof}

Let $0<p_1\leq p_2<1$ and consider the MT$(\mathbb{T},p_i)$-model, $i \in\{ 1, 2\}$. According to Lemma \ref{L:233}, for $i \in\{ 1, 2\}$, the respective rumor model survives if, and only if, its associated branching process $(Z_n^{(i)})_{n\geq 0}$ survives. We shall consider the following natural coupling between these branching processes. Let $(U_j)_{j \in \mathbb{N}}$ be a sequence of i.i.d. random variables such that $U_j \sim U(0,1)$ and define, for each $i\in\{1,2\}$, the sequence of i.i.d. random variables $\{I_{j}^{(i)}\}_{j\in\mathbb{N}}$ given by $I_j^{(i)} = \mathbb{1}_{\{U_j \leq p_i\}}, j\in\mathbb{N}.$ Using the sequence $(U_j)_{j \in \mathbb{N}}$, by \eqref{eq:XX} and \eqref{eq:Zn}, it is not difficult to see that we can construct both processes $(Z_n^{(i)})_{n\geq 0}$, $i\in\{1,2\}$, on the same probability space. Moreover, since $p_1 \leq p_2$, we have $I_j^{(1)} \leq I_j^{(2)}$, for all $j$, and by construction $ Z_n^{(1)} \leq Z_n^{(2)}$, for all $n \geq 0$. Therefore, for each $n \geq 1$, $\left\{Z_n^{(1)} \geq 1 \right\} \subseteq \left\{Z_n^{(2)} \geq 1 \right\},$ and consequently,

$$
\bigcap_{n \geq 1} \left\{Z_n^{(1)} \geq 1 \right\} \subseteq \bigcap_{n \geq 1} \left\{Z_n^{(2)} \geq 1 \right\},
$$
which in turns implies, by Lemma \ref{L:233}, that:

$$
\theta(p_1, d) = \mathbb{P} \left( \bigcap_{n \geq 1} \left\{Z_n^{(1)} \geq 1 \right\} \right) \leq \mathbb{P} \left( \bigcap_{n \geq 1} \left\{Z_n^{(2)} \geq 1 \right\} \right) = \theta(p_2, d).
$$

\end{proof}

\subsubsection{Proof of Theorem \ref{thm:MTp}}

Let $p\in (0,1], d\geq 3,$ and consider the MT$(\mathbb{T}_d,p)$-model with the standard initial configuration. By Lemma \ref{L:233}, $\theta(d,p)>0$ if, and only if, the branching process starting with $N'$ particles, and having an offspring distribution given by $X'$ survives with positive probability. It is a well-known result of the Theory of Branching Processes, see \cite[Theorem 1.1, Chapter 2]{schinazi}, that the last happens if, and only, if $\mathbb{E}(X')>1$. Thus, by Lemma \ref{lem:EXl}, together with \eqref{eq:meanX}, we conclude that $\theta(d,p)>0$ if, and only if,
$$p>\left\{\frac{d e^{d+1}}{(d+1)^{d}}\Gamma(d,d+1)\right\}^{-1}.$$
Therefore we obtain \eqref{eq:pcdef}. Note that it is directly verified that $p_c(d)>0$, while $p_c(d)<1$ is a consequence of Lemma \eqref{lem:AdaltoEX}. Now we obtain the asymptotic expression in \eqref{eq:pcasymp} for $p_c(d)$. 
By Lemma \ref{lem:gamma}(ii) we obtain

$$
 \displaystyle \frac{d e^{d+1}}{(d+1)^{d}}\Gamma(d,d+1)
 \sim  \frac{d e^{d+1}}{(d+1)^{d}}\left(\frac{d}{e}\right)^d \sqrt{\frac{\pi}{2d}} \sim e \sqrt{\frac{\pi d}{2}} \left(\frac{d}{d+1}\right)^d,
$$

\noindent
but $\left\{d/(d+1)\right\}^d \sim e^{-d}$. Therefore $$p_c(d) \sim \sqrt{\frac{2}{\pi d}}.$$

\subsubsection{Proof of Theorem \ref{T:234}}

Consider the  MT$(\mathbb{T}_d, p)$-model with the standard initial configuration, and let $\theta(d,p)$ be the probability of survival of the rumor. By Lemma \ref{L:233} $\theta(d,p)$ is the probability of survival of a branching process starting with $N'$ particles at time $0$ and having offspring distribution according to the random variable $X'$ with probability generating function given by Lemma \ref{lem:EXl}. Moreover, $\theta(d,p)=1-\mathbb{P}(\mathcal{E}),$ where $\mathcal{E}$ denotes the event of extinction of the branching process. Note that,

\begin{equation}\label{eq.E5}
    \mathbb{P}(\mathcal{E}) =\displaystyle\sum_{i=0}^{d+1}\mathbb{P}(E|{N'}=i)\mathbb {P}({N'}=i),
\end{equation}
and that the extinction of the branching process occurs, if and only, if all $N'$ independent and identically distributed branching processes starting at the first generation die out. Thus, $\mathbb{P}(\mathcal{E}\mid N'=i)=\psi^{i}$, where $\psi$ is the smallest nonnegative root of $G_{X'}(s)=s$, which by Lemma \ref{lem:EXl} turns into

\begin{align}
\label{eq:generating-eq}
s &= \frac{d}{d+1}\, \left(\frac{sp + 1 - p}{d + 1}\right)^{d-1} e^{\frac{d + 1}{sp + 1 - p}}\, \Gamma\left(d, \frac{d + 1}{sp + 1 - p} \right) p(s-1) + 1.
\end{align}

Now, by the previous remarks, \eqref{eq.E5}, and  \eqref{eq:lawN'}, we obtain

$$
     \mathbb{P}(\mathcal{E})  =\displaystyle\sum_{i=0}^{d+1} \left( \dfrac{p\psi}{1-p} \right)^{i}  \dfrac{1}{d+1} \sum_{k = i}^{d+1} k\,k! \binom{k}{i} \binom{d+1}{k} \left( \dfrac{1-p}{d+1} \right)^{k},
     $$
where $\psi$ is the smallest nonnegative root of \eqref{eq:generating-eq}. Therefore, 
\[
\theta(d) = 1 - \frac{1}{d+1} \displaystyle\sum_{i=0}^{d+1} \left( \dfrac{p\psi}{1-p} \right)^{i} \,  \sum_{k = i}^{d+1}k\,k! \binom{k}{i} \binom{d+1}{k} \left( \dfrac{1-p}{d+1} \right)^{k}.
\]

\subsubsection{Proof of Theorem \ref{theo:princ}}

Let $\alpha\in (0,1], d\geq 3,$ $k<d$ and consider the MT$(\mathbb{T},d,k,\alpha,h)$-model with the standard initial configuration. We can follow a branching process construction similar to the one used in the proof of Theorem~\ref{thm:MTp}. The main difference is that we now focus on the hubs instead of all the vertices. Roughly speaking, the branching process starts from the root (which is a hub). The children of this initial particle are the hubs at distance $h$ that receive the rumor from $\bf 0$, if any. These hubs form the first generation of the branching process. The second generation consists of those hubs, at distance $h$ from the previous ones (away from the root), which also receive the rumor, and so on. If we denote by $Z_n$ the number of hubs reached by the rumor at distance $nh$ from the root, then the resulting stochastic process $(Z_n)_{n\geq 0}$ is a branching process. The survival of this branching process is equivalent to the survival of the rumor in the MT$(\mathbb{T},d,k,\alpha,h)$-model with the standard initial configuration. The offspring distribution of this branching process is obtained by observing a spreader, different of the root, and its non-spreader neighbors identified with labels from $1$ to $d$. Recall that as soon as an ignorant becomes a spreader in a tree, in the model starting with the standard initial configuration, it has one neighbor already in the spreader state (the one who informed it) and $d$ neighbors in the ignorant state. Thus, the offspring distribution of the branching process is given by
$$\sum_{i=1}^{X} Y_i$$
where $X$ is the number of spreaders that a spreader, different of the root, generates, which law and mean are stated in Lemma \ref{lem:AdaltoEX} and Lemma \ref{lemma:224}, respectively, and $Y_i$ is an indicator random variable. $Y_i$ is associated to the $ith-$non-spreader neighbor of the hub, indicating the event that such a vertex is connected to another hub, with happens with probability $\alpha$, and the rumor flows from this vertex to the other hub, with happens with probability $\beta(k-1)^{h-1}$. Here $\beta(k-1)$ is the probability that a spreader, in a path between two hubs, contacts its nearest neighbor in the direction to a hub and away from the root, before becoming stifler. In this case, such vertex has $k$ neighbors. By the Wald's equation, see \cite[Proposition 11.4]{ross}, the previous remarks, Lemma \ref{lem:beta}, and Theorem \ref{thm:MTp}, we have that
$$\mathbb{E}\left(\sum_{i=1}^{X} Y_i\right)=\mathbb{E}(X) \alpha \beta(k-1)^{h-1}>1,$$
if, and only if, 
$$\alpha >\frac{1}{\mathbb{E}(X)} \beta(k-1)^{1-h} = p_c(d)\left\{\frac{e^{k}\Gamma(k-1,k)-\Gamma(k-1)}{k^{k-1}}\right\}^{1-h}.$$

\subsubsection{Proof of Corollary \ref{cor:1special}}

Let $\alpha\in (0,1], d\geq 3,$ $k=o(d)$ and consider the MT$(\mathbb{T},d,k,\alpha,h)$-model with the standard initial configuration. Then $\alpha_c(d,k,h)\in(0,1)$ if, and only if, $p_c(d)\beta(k-1)^{1-h}<1$ which, in turns, is equivalent to
\begin{equation}\label{eq:hineq}
h < \frac{\log p_c(d)}{\log \beta(k-1)}+1.
\end{equation}

By Theorem \ref{thm:MTp} and Lemma \ref{lem:beta}, applying their asymptotic estimates for large $d$, the inequality \eqref{eq:hineq} reduces to

$$
h \lesssim \frac{\log(2/\pi)-\log d}{\log(\pi/2)-\log(k-1)} \sim \frac{\log d}{\log k},
$$
which completes the proof.
{ 
\section{Conclusion}

Nonlinear dynamics and nonequilibrium statistical physics provide a natural framework for modeling rumor propagation in complex networks, as these processes are inherently stochastic and highly sensitive to fluctuations. In this work, we study the use of special stochastic processes to represent rumor propagation on trees. We develop arguments that establish a phase transition result for an extension of the well-known Maki–Thompson rumor model on an infinite Cayley tree and identify the corresponding critical threshold. Our approach relies on comparing the original model with a suitably defined branching process. The methods we use are constructive and can be easily adapted to more ``realistic'' models. Motivated by rumor spreading in Barabási–Albert–type networks, we also analyze the Maki–Thompson model on random trees formed by three types of vertices: hubs, each connected to $d+1$ other vertices; regular vertices, each connected to $1<k=o(d)$ vertices; and leaves. Assuming that each hub is, on average, connected to $\alpha(d+1)$ other hubs through paths of length $h$, with $\alpha\in(0,1]$, we prove a phase transition result in $\alpha$ that depends on $d$, $k$, and $h$. We further illustrate our results in the specific case where $k$ is of order $\log d$.

Our findings are novel and not covered in existing literature. They extended previous studies to a broader framework. Importantly, all of our results are rigorously proven and include the exact localization of critical thresholds through closed-form expressions. This offers an alternative to methods based on symbolic computation. While our work mainly contributes to the field of theoretical models for rumor spreading, we also expect it to inspire new work involving comparisons with physics experiments and/or observations.}

\section*{CRediT authorship contribution statement}

Jhon F. Puerres: Conceptualization, Methodology, Formal analysis, Investigation, Writing - Original Draft, Writing - Review $\&$ Editing. Valdivino V. Junior: Methodology, Formal analysis. Writing - Review $\&$ Editing, Supervision. Pablo M. Rodriguez: Conceptualization, Methodology, Formal analysis, Investigation, Writing - Original Draft, Writing - Review $\&$ Editing, Supervision, Funding acquisition.

\section*{Declaration of competing interest}
We have nothing to declare.

\section*{Acknowledgements}

This work has been developed with the support of the Fundação Coordenação de Aperfeiçoamento de Pessoal de Nível Superior (CAPES), Financial Code 001, Conselho Nacional de Desenvolvimento Científico e Tecnológico (Grant 316121/2023-1), Fundação de Amparo à Ciência e Tecnologia do Estado de Pernambuco - FACEPE (Grant APQ-1341-1.02/22, IBPG-0027-1.02/23), and Fundação de Amparo à Pesquisa do Estado de São Paulo - FAPESP (Grant 23/13453-5).


\begin{thebibliography}{00}

{ 
\bibitem{1aaa}
Wang GW, Tan ZX, Gao XY, Liu JG. A new $(2+1)-$dimensional like-Harry-Dym equation with derivation and soliton solutions. Appl Math Lett 2026;172:109720. \url{https://doi.org/10.1016/j.aml.2025.109720}.

\bibitem{1bbb} Gao XY, Liu JG, Wang GW. Inhomogeneity, magnetic auto-B\"{a}cklund transformations and magnetic solitons for a generalized variable-coefficient Kraenkel-Manna-Merle system in a deformed ferrite. Appl Math Lett 2025;171:109615. \url{https://doi.org/10.1016/j.aml.2025.109615}.

\bibitem{1ccc}
Tang Y, Liu J, Zhang J, Zhang P. Learning nonequilibrium statistical mechanics and dynamical phase transitions. Nat Commun 2024;15:1117. \url{https://doi.org/10.1038/s41467-024-45172-8}

\bibitem{1ddd}
 Krapivsky PL, Redner S, Ben-Naim E. A Kinetic View of Statistical Physics. Cambridge University Press; 2010. \url{
    https://doi.org/10.1017/CBO9780511780516}

    \bibitem{1eee}
    Draief M, Massoulié L. Epidemics and Rumours in Complex Networks. Cambridge University Press; 2009. \url{
    https://doi.org/10.1017/CBO9780511806018}

    \bibitem{1fff}
    Andersson H, Britton T. Stochastic Epidemic Models and Their Statistical Analysis. Volume 151. New York: Springer; 2000. \url{https://doi.org/10.1007/978-1-4612-1158-7} 

}

\bibitem{daley_nature}
Daley DJ, Kendall DG. Epidemics and rumours. Nature 1964;204:1118. \url{http://dx.doi.org/10.1038/2041118a0}.

\bibitem{moreno-PhysA2007}
Nekovee M, Moreno Y, Bianconi G, Marsili M. Theory of rumour spreading in complex social networks. Physica A 2007;374:457-70. \url{http://dx.doi.org/10.1016/j.physa.2006.07.017}.

\bibitem{MNP-PRE2004}
Moreno Y, Nekovee M, Pacheco AF. Dynamics of rumour spreading in complex networks. Phys Rev E 2004;69:066130. \url{http://dx.doi.org/10.1103/PhysRevE.69.066130}.

\bibitem{Pastor}
Pastor-Satorras R, Castellano C, Van Mieghem P, Vespignani A. Epidemic processes in complex networks. Rev Mod Phys 2015;87(3):925–79. \url{http://dx.doi.org/10.1103/RevModPhys.87.925}.

\bibitem{zhang}
Zhang W, Brandes U. Conformity versus credibility: A coupled rumor-belief model. Chaos Solitons Fractals 2023;176:114172. \url{http://dx.doi.org/10.1016/j.chaos.2023.114172}.

\bibitem{zhou}
Zhou Q, Duan X, Yu G. Research on dynamic modeling and control mechanisms of rumor spread considering high-order interactions and counter-rumor groups. Chaos Solitons Fractals 2025;197:116498. \url{https://doi.org/10.1016/j.chaos.2025.116498}.

\bibitem{DK}
Daley DJ, Kendall DG. Stochastic rumours. J Inst Math Appl 1965;1:42-55. \url{http://dx.doi.org/10.1093/imamat/1.1.42}.


\bibitem{MT}
Maki DP, Thompson M. Mathematical models and applications. With emphasis on the social, life, and management sciences. Englewood Cliffs (NJ): Prentice-Hall; 1973.

\bibitem{LTRM}
Lebensztayn E, Machado FP, Rodriguez PM. Limit theorems for a general stochastic rumour model. SIAM J Appl Math 2011;71:1476-86. \url{http://dx.doi.org/10.1137/100819588}.


\bibitem{GR}
Grejo C, Rodriguez PM. Asymptotic behavior for a modified Maki–Thompson model with directed inter-group interactions. J Math Anal Appl 2019;480:123402. \url{http://dx.doi.org/10.1016/j.jmaa.2019.123402}.


\bibitem{RPRS}
Lebensztayn E, Machado FP, Rodriguez PM. On the behaviour of a rumour process with random stifling. Environ Model Softw 2011;26:517-22. \url{http://dx.doi.org/10.1016/j.envsoft.2010.10.015}.


\bibitem{MRR}
Rada A, Coletti C, Lebensztayn E, Rodriguez PM. The role of multiple repetitions on the size of a rumor. SIAM J Appl Dyn Syst 2021;20:1209-31. \url{http://dx.doi.org/10.1137/20M1345657}.


\bibitem{EAP}
Agliari E, Pachon A, Rodriguez PM, Tavani F. Phase transition for the Maki-Thompson rumor model on a small-world network. J Stat Phys 2017;169:846-75. \url{http://dx.doi.org/10.1007/s10955-017-1892-x}.


\bibitem{Diaz-rodriguez-rua/2025}
Diaz Bacca AC, Rodriguez PM, Rúa-Alvarez C. The optimal degree for maximizing rumor spreading on a ring lattice. arXiv preprint arXiv:2507.02141. 2025. \url{http://dx.doi.org/10.48550/arXiv.2507.02141}.

\bibitem{junior-rodriguez-speroto/2020}
Junior VV, Rodriguez PM, Speroto A. The Maki–Thompson rumor model on infinite Cayley trees. J Stat Phys 2020;181:1204-17. \url{http://dx.doi.org/10.1007/s10955-020-02623-y}.


\bibitem{junior-rodriguez-speroto/2021}
Junior VV, Rodriguez PM, Speroto A. Stochastic rumors on random trees. J Stat Mech 2021;2021:123403. \url{http://dx.doi.org/10.1088/1742-5468/ac3b45}.

\bibitem{Barabasi1999}
Barabási AL, Albert R. Emergence of scaling in random networks. Science 1999;286(5439):509–12. \url{http://dx.doi.org/10.1126/science.286.5439.509}.

\bibitem{Cohen2003}
Cohen R, Havlin S. Scale-free networks are ultra small. Phys Rev Lett 2003;90:058701. \url{http://dx.doi.org/10.1103/PhysRevLett.90.058701}.

\bibitem{Albert1999}
Albert R, Jeong H, Barabási AL. Diameter of the world-wide web. Nature 1999;401(6749):130–131. \url{http://dx.doi.org/10.1038/43601}.


\bibitem{jameson}
Jameson GJO. The incomplete gamma functions. Math Gaz 2016;100:298-306. \url{http://dx.doi.org/10.1017/mag.2016.67}.

\bibitem{ross}
Ross S. Introduction to probability models. 10th ed. Amsterdam: Academic Press; 2010. 


\bibitem{schinazi}
Schinazi RB. Discrete time branching process. In: Classical and spatial stochastic processes. With applications to biology. 2nd ed. New York: Birkhäuser; 2014. p. 18-46. 


\bibitem{gomez-junior-rodriguez/2024}
Gomez LM, Junior VV, Rodriguez PM. The impact of effective participation in stopping misinformation: an approach based on branching processes. J Stat Mech 2024;2024:033402. \url{http://dx.doi.org/10.1088/1742-5468/ad2677}.




\bibitem{flajolet-gardy-thimonier/1992}
Flajolet P, Gardy D, Thimonier L. Birthday paradox, coupon collectors, caching algorithms and self-organizing search. Discrete Appl Math 1992;39:207-29. \url{http://dx.doi.org/10.1016/0166-218X(92)90177-C}.



\bibitem{isaac}
Isaac R. The coupon collector’s problem solved. In: The pleasures of probability. Undergraduate Texts in Mathematics. New York: Springer; 1995. p. 80-2. 


\bibitem{motwani}
Motwani R, Raghavan P. The coupon collector’s problem. In: Randomized algorithms. Cambridge: Cambridge Univ Press; 1995. p. 57-63. 



\end{thebibliography}
\end{document}